\documentclass{article}
\usepackage{graphicx} 
\usepackage{enumitem}

\usepackage{amsmath,amssymb,amsthm,mathrsfs,cite}
\usepackage{geometry}
\usepackage{centernot}
\usepackage{tikz}
\usepackage{pgfplots}
\pgfplotsset{compat=1.18}

\newtheorem{theorem}{Theorem}
\newtheorem{definition}{Definition}

\newtheorem{lemma}[theorem]{Lemma}

\newtheorem{question}[theorem]{Question}

\newcommand{\R}{\mathbb{R}}

\newcommand{\ip}[2]{\left\langle #1,#2 \right\rangle}
\newcommand{\br}[2]{\langle #1,#2\rangle}
\newcommand{\norm}[1]{\left\lVert #1\right\rVert}

\title{A Characterization of Quadrics Among Affine Hyperspheres by Section-Centroid Location}
\author{Alexandre Borentain}

\date{}

\begin{document}

\maketitle

\begin{abstract}
A theorem of Meyer and Reisner characterizes ellipsoids by the collinearity of centroids of parallel sections: if $\Omega\subset\R^{n+1}$ is a convex body such that for every $n$-dimensional subspace $M\subset\R^{n+1}$ the centroids of the sections $(x+M)\cap \Omega$ are collinear, then $\Omega$ is an ellipsoid. 

We study natural extensions of this centroid–collinearity condition to unbounded convex sets. In particular, we show that among affine hyperspheres, precisely the ellipsoids, paraboloids and one sheet of a two-sheeted hyperboloid satisfy this property. We also identify additional assumptions under which any convex hypersurface with this property must necessarily be a quadric.
\end{abstract}

\section{Introduction and main results}

Let $\Omega\subset\mathbb{R}^{n+1}$ be a convex body. A classical theorem of Blaschke asserts that if, for every family of parallel chords of $\Omega$, their midpoints lie in a common hyperplane, then $\Omega$ is an ellipsoid. Meyer and Reisner later study the following property: for every $n$–dimensional subspace $M\subset\mathbb{R}^{n+1}$ such that $(x+M)\cap\Omega$ is bounded, the centroids of the sections $(x+M)\cap \Omega$ are collinear (as $x$ ranges over all translates for which the section meets $\operatorname{int}\Omega$). We refer to this as the \emph{section--centroid collinearity} property (SCCP). Meyer and Reisner prove that a convex body with SCCP must be an ellipsoid \cite{Meyer}. 

\medskip

This paper investigates which unbounded convex sets satisfy the SCCP. In particular, we show that among affine hyperspheres, precisely the ellipsoids, paraboloids and one sheet of a two-sheeted hyperboloid satisfy the SCCP. Finally, we explore some constructions that aim to better understand the conditions under which surfaces with SCCP are affine hyperspheres. 

A surface is an affine hypersphere if its affine normal lines intersect in a single point. Affine hyperspheres were introduced by \c{T}i\c{t}eica \cite{Tzitzeica1907,Tzitzeica1908,Tzitzeica1909}
and were later studied by Blaschke \cite{Blaschke1923}, Calabi \cite{Calabi1972CompleteAffineHyperspheresI}, Cheng--Yau \cite{ChengYau1986CompleteAffineHyperspheresI}, and many others. A basic classification according to affine
mean curvature may be summarized as follows: affine hyperspheres with positive affine mean curvature
are ellipsoids; those with zero affine mean curvature are elliptic paraboloids; and the class with
negative affine mean curvature is much richer, since Cheng and Yau established that for any proper convex cone there exists a complete hyperbolic affine hypersphere
asymptotic to its boundary \cite{ChengYau1986CompleteAffineHyperspheresI,ChengYau1977MongeAmpereRegularity}. Identifying quadrics within this last class has been of interest; for
example, Pick and Berwald proved that vanishing Pick form implies that the hypersphere is locally an
open subset of a quadric; see \cite{NomizuSasaki1994}.

A recurring theme is the classical relationship between centroids of parallel sections and the \emph{affine normal vector}. Infinitesimally, the affine normal line at a point of a smooth strictly convex hypersurface is tangent to the locus of centroids of sections parallel to the tangent plane; thus, when the family of section centroids in a fixed normal direction is collinear, its direction agrees with the affine normal direction. This provides the bridge between centroid geometry and equi-affine surface theory that we exploit below.

Our main result is the following classification among affine hyperspheres.

\begin{theorem}\label{thm:main}
Let $\Omega\subset\R^{n+1}$ be a strictly convex domain with smooth boundary $\partial\Omega$, equipped with its Blaschke (equi-affine) normalization, so that $\partial\Omega$ is an affine hypersphere. Assume that for every $n$--dimensional linear subspace $M\subset\R^{n+1}$ such that the section $(x+M)\cap\Omega$ is bounded, the centroids of the sections $(x+M)\cap\Omega$ are collinear as $x$ varies over all translates for which the section meets $\operatorname{int}\Omega$. Then $\partial\Omega$ is one of the following quadrics:
\begin{itemize}
  \item an ellipsoid,
  \item a paraboloid, or
  \item one sheet of a two-sheeted hyperboloid.
\end{itemize}
\end{theorem}

This result follows from adapting the techniques used by Meyer–Reisner \cite{Meyer} to the unbounded setting and using characterizations of quadric surfaces due to Kim \cite{kimHyp,kimPara}. We prove a small extension of this theorem in which we replace the affine hypersphere condition by
the weaker assumption of asymptotic convergence to a cone:
\begin{theorem}\label{thm:thm2}
Let $\Omega\subset\mathbb{R}^{n+1}$ be a strictly convex domain with smooth boundary $\partial\Omega$.
Assume that the recession cone of $\Omega$ is either
\begin{enumerate}[label=(\roman*),ref=(\roman*)] 
  \item $\{0\}$,
  \item $1$--dimensional, or
  \item $(n+1)$--dimensional; and in this case, assume in addition that $\partial\Omega$ is asymptotic to a cone.
\end{enumerate}
Suppose moreover that $\Omega$ satisfies the following:

\smallskip
\noindent
For every $n$--dimensional linear subspace $M\subset\mathbb{R}^{n+1}$ such that $(x+M)\cap\Omega$ is bounded, the centroids of the sections $(x+M)\cap\Omega$ are collinear as $x$ varies over all translates for which $(x+M)\cap\operatorname{int}\Omega\neq\varnothing$.
\smallskip

\noindent
Then $\partial\Omega$ is a quadric, and more precisely:
\begin{itemize}
  \item in case \emph{(i)} $\partial\Omega$ is an ellipsoid;
  \item in case \emph{(ii)} $\partial\Omega$ is a paraboloid;
  \item in case \emph{(iii)} $\partial\Omega$ is one sheet of a two-sheeted hyperboloid.
\end{itemize}
\end{theorem}

Theorem~\ref{thm:main} should be compared with the Meyer--Reisner theorem\cite{Meyer}: in their setting the convex set is bounded, and so $(x+M)\cap\Omega$ is always bounded. Conveniently, for convex unbounded sets, if $(y+M)\cap\Omega$ is bounded and nonempty, then $(x+M)\cap\Omega$ is also bounded for all $x\in \R^{n+1}$.

It is natural to ask whether the affine hypersphere assumption or asymptotic convergence to a cone can be removed.

\begin{question}\label{ques:question}
Let $\Omega\subset\R^{n+1}$ be a strictly convex domain with smooth boundary $\partial\Omega$, and suppose that for every $n$--dimensional subspace $M\subset\R^{n+1}$ such that $(x+M)\cap\Omega$ is bounded, the centroids of the sections $(x+M)\cap\Omega$ are collinear as $x$ varies over all translates with $(x+M)\cap\operatorname{int}\Omega\neq\varnothing$. Must $\partial\Omega$ be an ellipsoid, paraboloid, or one sheet of a two-sheeted hyperboloid?
\end{question}

\paragraph{Organization.}
In the rest of this section, we fix notation and recall basic material on recession cones, Gauss maps, and lines of centroids. In Section~\ref{sec:main} we adapt several lemmas of Meyer and Reisner to the unbounded case and derive a volume--cut functional whose level sets determine our initial set under some homothety. We connect this to work in \cite{kimHyp} and \cite{kimPara} to prove Theorem~\ref{thm:main}. In Section~\ref{sec:extension} we investigate the behavior of centroid lines using recession cones, and prove Theorem~\ref{thm:thm2}. Finally, in the last section we discuss the remaining open cases needed for a complete classification of convex hypersurfaces with SCCP.

\subsection{Notation and basic objects}

Throughout, $\langle\cdot,\cdot\rangle$ denotes the standard inner product on $\R^{n+1}$, $\|\cdot\|$ the Euclidean norm, and $\mathcal{H}^k$ the $k$--dimensional Hausdorff measure. The unit sphere in $\R^{n+1}$ is denoted $S^{n}$. The interior and boundary of a set $A$ are denoted $\operatorname{int}A$ and $\partial A$, respectively.

\begin{definition}[Hyperplanes, sections, and centroids]
For $u\in S^{n}$ and $t\in\R$, define the hyperplane
\[
\Pi(u,t):=\{x\in\R^{n+1}:\ \langle u,x\rangle=t\}.
\]
For any $X\subset\R^{n+1}$ define the section
\[
\Sigma(u,t;X):=\Pi(u,t)\cap X.
\]
When $0<\mathcal{H}^{n}(\Sigma(u,t;X))<\infty$, the centroid of the section is
\[
\mathrm{cen}(u,t;X)
:=\frac{1}{\mathcal{H}^{n}(\Sigma(u,t;X))}
  \int_{\Sigma(u,t;X)} x\,\mathrm d\mathcal{H}^{n}(x)\in\Pi(u,t).
\]
\end{definition}

\begin{definition}[Minkowski functional and support function]
For a convex body $K\subset\R^{n+1}$ with $0\in\operatorname{int}K$, the Minkowski functional is
\[
\|x\|_K:=\inf\{\lambda>0:\ x\in \lambda K\},
\]
and the support function is
\[
h_K(u):=\sup_{y\in K}\langle u,y\rangle\quad(u\in\R^{n+1}).
\]
\end{definition}

\begin{definition}[Recession cone]
For a convex (not necessarily bounded) set $\Omega\subset\R^{n+1}$, the recession cone is
\[
\operatorname{rec}(\Omega):=\{d\in\R^{n+1}:\ x+td\in\Omega\ \ \forall x\in\Omega,\ t\ge 0\}.
\]
It is a closed convex cone, and we denote its linear dimension by $\dim\operatorname{rec}(\Omega)$.
\end{definition}

\subsubsection{Convergence to cones}

We now introduce two different notions of convergence of an unbounded convex hypersurface to a cone which will be useful to us. Let $X\subset\R^{n+1}$ be a closed, strictly convex set with smooth boundary $\partial X$ an $n$--dimensional hypersurface, and let $C\subset\R^{n+1}$ be the boundary of a closed convex cone with apex at the origin, i.e., $\lambda C=C$ for all $\lambda>0$.

For $R>0$, write $S_R:=\{x\in\R^{n+1}:\|x\|=R\}$.

\begin{definition}[Blow-down convergence]\label{def:blowdown}
Let $X\subset\R^{n+1}$ be a nonempty closed convex set with nonempty interior, and fix some $x_0\in X$.
The \emph{blow-down cone} of $X$ is the Painlev\'e--Kuratowski limit
\[
X_\infty \;:=\; \lim_{t\to\infty} t^{-1}(X - x_0),
\]
whenever this limit exists.
\end{definition}

It is well known that for closed convex sets this limit always exists, is independent of the choice of $x_0$, and coincides with the classical recession cone $\operatorname{rec}(X)$, see for instance \cite[Chap.~2]{AuslenderTeboulle2003} or \cite{GarciaGoicocheaLopezMartinez2025}.

\begin{definition}[Asymptotic convergence]\label{def:conv}
We say that \emph{$\partial X$ is asymptotic to $C$} if
\[
d\big(\partial X\cap S_R,\; C\cap S_R\big)\xrightarrow[R\to\infty]{}0.
\]
\end{definition}

This definition is the one used in work on the Calabi conjecture \cite{Calabi1972} and \cite{LiSimonZhao1993} and is not to be confused with the asymptotic cone which is equivalent to the recession cone.

\medskip

If $\partial X$ is asymptotically convergent to a cone $\partial\mathcal{C}$ then this cone must be $\partial\operatorname{rec}(X)$. So the asymptotic and blow-down limits are equal when the former exists.

It is possible, however, for a convex set to have a well-defined blow-down cone (equivalently, a recession cone) while its boundary fails to be asymptotic to any cone in the sense of Definition~\ref{def:conv}. 
For example, the epigraph of an elliptic paraboloid in $\R^{n+1}$ has recession cone equal to a single ray. 

More generally, blow-down convergence alone does not force asymptotic convergence of $\partial X\cap S_R$ to $\operatorname{rec}(X)\cap S_R$. 
Conversely, even when $\operatorname{rec}(X)$ is full-dimensional, asymptotic convergence may still fail: see Figure~\ref{fig:asymptotic-parabola}, where the epigraph of $f(x)=e^x$ has recession cone equal to the second quadrant, but the graph of $f$ is not asymptotic to the boundary of that cone.

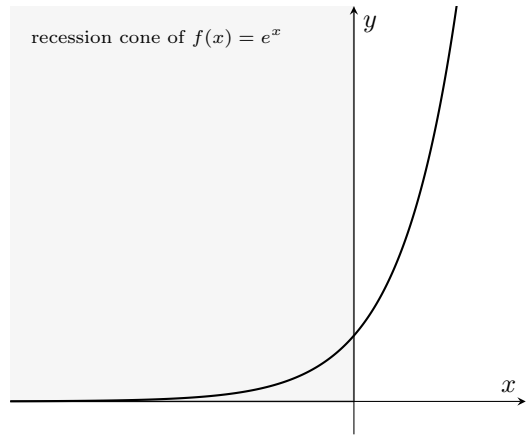
\begin{figure}[ht]
    \centering
    \begin{tikzpicture}
        \begin{axis}[
            axis lines = middle,
            xlabel = {$x$},
            ylabel = {$y$},
            xmin = -6, xmax = 3,
            ymin = -0.5, ymax = 6,
            xtick=\empty, ytick=\empty, 
            axis on top=true,
        ]

            \path[fill=gray!20, opacity=0.35]
                (axis cs:-6,0) rectangle (axis cs:0,6);

            \addplot[gray!60, thick] coordinates {(-6,0) (0,0)};
            \addplot[gray!60, thick] coordinates {(0,0) (0,6)};

            \addplot[
                domain = -6:3,
                samples = 200,
                smooth,
                thick
            ] {exp(x)};

            \node[anchor=north west, font=\scriptsize, align=left]
              at (axis cs:-5.8,5.8)
              {recession cone of $f(x)=e^x$};

        \end{axis}
    \end{tikzpicture}
    \caption{The function $f(x)=e^x$.}
    \label{fig:asymptotic-parabola}
\end{figure}

\subsubsection{Gauss map and centroid loci}

Let $\Omega\subset\R^{n+1}$ have $C^2$ strictly convex boundary $\partial\Omega$. The (outer) Euclidean Gauss map
\[
N:\partial\Omega\to S^{n},\qquad N(x)=\text{outer unit normal at }x,
\]
is well defined and a homeomorphism onto its image. The image $U:=N(\partial\Omega)\subset S^{n}$ will be the set of admissible normals.

\begin{definition}[Centroid curve and centroid line]
Fix $u\in U$. Define
\[
I(u):=\left\{t\in\R:\ 0<\mathcal{H}^{n}(\Sigma(u,t;\Omega))<\infty\right\},
\]
and the \emph{centroid curve}
\[
\gamma_u\ :=\ \{\mathrm{cen}(u,t;\Omega):\ t\in I(u)\}\ \subset \R^{n+1}.
\]
When $\gamma_u$ is contained in some affine line $\ell_u$, we call $\ell_u$ the \emph{centroid line in the direction $u$}. When the underlying set is important we write $\ell_u^\Omega$.
\end{definition}

\begin{definition}[Cut--volume functional]
For $a\in\R^{n+1}\setminus\{0\}$, write
\begin{align*}
    &H(a):=\{x\in\R^{n+1}:\ \langle a,x\rangle=1\},\\
    &C(a)^{+}:=\{x\in\R^{n+1}:\ \langle a,x\rangle\ge 1\},\\
    &C(a)^{-}:=\{x\in\R^{n+1}:\ \langle a,x\rangle\le 1\},
\end{align*}
and, for a measurable $\Omega\subset\R^{n+1}$,
\[
V(a):=\mathcal{H}^{n+1}\bigl(\Omega\cap C(a)^{-}\bigr).
\]
And denote by $x(a)$ the centroid of the section $H(a)\cap \Omega$.

\end{definition}

Note that 
\[
H(a)\cap\Omega=\Sigma\ \bigg(\frac{a}{\norm{a}},\frac{1}{\norm{a}}\bigg),\quad \text{and} \quad x(a)=\gamma\ \bigg(\frac{a}{\norm{a}},\frac{1}{\norm{a}}\bigg).
\]

\section{Core lemmas and classification results}\label{sec:main}

In this section we adapt two lemmas from \cite{Meyer} to unbounded convex sets in $\R^{n+1}$ and derive a key consequence: under SCCP, certain cut volumes determined by support hyperplanes are constant. This will be the main tool in the classification of affine hyperspheres with SCCP. The adapted proofs are nearly identical to the originals but we include them for completeness.

\subsection{Differentiability of the cut-volume functional}

The following lemma is a variant of Lemma~5 in \cite{Meyer}, adapted to cuts by half-spaces of the form $C(a)^-$.

\begin{lemma}\label{gradV}
Let $\Omega\subset\R^{n+1}$ be a unbounded convex set with $0\notin \Omega$ and smooth boundary. Then $V$ is $C^1$ at every $a$ for which $0<V(a)<\infty$, and
\[
x(a)=\frac{\nabla V(a)}{\br{a}{\nabla V(a)}}.
\]
\end{lemma}

\begin{proof}
Since $\Omega\cap C(a)^-$ is convex with nonempty interior and
$0<V(a)<\infty$, it must be bounded. In particular, $H(a)\cap \Omega$
is bounded and has finite $n$--dimensional measure.

Let $\operatorname{rec}(\Omega)$ denote the recession cone of $\Omega$. If
$v\in\operatorname{rec}(\Omega)\setminus\{0\}$ and $\ip{a}{v}\le 0$, then for any
$x_0\in \Omega\cap C(a)^-$ we would have $x_0+tv\in \Omega\cap C(a)^-$ for all $t\ge0$,
so $\Omega\cap C(a)^-$ would be unbounded, a contradiction. Hence
\begin{equation}\label{eq:rec-positive}
\ip{a}{v}>0\qquad\forall\,v\in\operatorname{rec}(\Omega)\setminus\{0\}.
\end{equation}
By compactness of $\operatorname{rec}(\Omega)\cap S^{n}$ there exists $\alpha>0$
such that
\begin{equation}\label{eq:rec-alpha}
\ip{a}{v}\ge\alpha>0\qquad\forall\,v\in\operatorname{rec}(\Omega)\cap S^{n}.
\end{equation}

It follows that there exist $R>0$ and $\varepsilon_0>0$ such that for all
$\varepsilon$ with $|\varepsilon|\le\varepsilon_0$ and every coordinate
direction $e_j$,
\begin{equation}\label{eq:uniform-ball}
\Omega\cap C(a+\varepsilon e_j)^-\subset B_R(0).
\end{equation}
Indeed, if \eqref{eq:uniform-ball} failed, there would be a sequence
$x_k\in \Omega\cap C(a+\varepsilon_k e_j)^-$ with $\|x_k\|\to\infty$ and
$\varepsilon_k\to0$, and after normalizing we would obtain a direction
$d\in\operatorname{rec}(\Omega)\cap S^{n}$ with $\ip{a}{d}\le0$, contradicting
\eqref{eq:rec-positive}. Thus all cuts $\Omega\cap C(a+\varepsilon e_j)^-$ for
small $\varepsilon$ lie in a fixed ball $B_R(0)$.

\medskip

\emph{Step 1: One-sided derivative in a coordinate direction.}
Consider
\[
D_j(\varepsilon):=V(a+\varepsilon e_j)-V(a),\qquad\varepsilon\in\R.
\]
We compute the derivative at $\varepsilon=0$ for $j=1$; the other coordinates
are analogous, and the argument for $\varepsilon<0$ is the same as for
$\varepsilon>0$.

Write
\[
\Omega_+(\varepsilon):=\{x\in \Omega:\ \ip{a}{x}>1,\ \ip{a+\varepsilon e_1}{x}\le1\},
\]
\[
\Omega_-(\varepsilon):=\{x\in \Omega:\ \ip{a}{x}\le1,\ \ip{a+\varepsilon e_1}{x}>1\},
\]
so that
\[
D_1(\varepsilon)=|\Omega_+(\varepsilon)|-|\Omega_-(\varepsilon)|.
\]
Both $\Omega_+(\varepsilon)$ and $\Omega_-(\varepsilon)$ lie in the slab between the
hyperplanes $H(a)$ and $H(a+\varepsilon e_1)$, hence in the bounded set
$\Omega\cap\bigl(C(a)^-\cup C(a+\varepsilon e_1)^-\bigr)\subset B_R(0)$ for
$|\varepsilon|\le\varepsilon_0$ by \eqref{eq:uniform-ball}.

Let $P$ denote orthogonal projection onto $H(a)$, and
$Q_\varepsilon x$ the projection of $x$ onto
$H(a+\varepsilon e_1)$, parallel to the vector $a$. Then, we have for $x\in H(a)$:
\begin{equation}\label{eq:segment-length}
\|x-Q_\varepsilon x\|
=\frac{|\varepsilon||x_1|}{\|a\|+\varepsilon\frac{a_1}{\|a\|}}\,
\end{equation}

Define
\[
U(\varepsilon):=H(a)\cap \Omega\cap P\bigl(H(a+\varepsilon e_1)\cap \Omega\bigr),\qquad
W(\varepsilon):=\{x\in H(a):\ [x,Q_\varepsilon x]\cap \Omega\neq\emptyset\}.
\]
Each point of $\Omega_+(\varepsilon)$ lies on a segment $[x,Q_\varepsilon x]$
with $x\in U(\varepsilon)$, and conversely each such segment contributes a
sliver of $\Omega_+(\varepsilon)$ between the two hyperplanes. By Fubini's theorem

\begin{equation}\label{eq:Omegaplus-bounds}
\frac{\varepsilon}{\|a\|+\varepsilon\frac{a_1}{\|a\|}}
   \int_{U(\varepsilon)\cap\{x_1\ge0\}} x_1\,d\sigma(x)
\;\le\;
|\Omega_+(\varepsilon)|
\;\le\;
\frac{\varepsilon}{\|a\|+\varepsilon\frac{a_1}{\|a\|}}
   \int_{W(\varepsilon)\cap\{x_1\ge0\}} x_1\,d\sigma(x),
\end{equation}
An analogous estimate holds for $\Omega_-(\varepsilon)$ with $x_1\le0$.

Because $\Omega\cap C(a)^-$ is compact, the sets $H(a)\cap \Omega$ and
$H(a+\varepsilon e_1)\cap \Omega$ converge in the Hausdorff metric as
$\varepsilon\to0$, and hence
\[
U(\varepsilon),W(\varepsilon)\xrightarrow[\varepsilon\to0]{d_H} H(a)\cap \Omega.
\]
Since $H(a)\cap \Omega$ is compact, the function $x\mapsto x_1$ is bounded on
all these sets. Letting $\varepsilon\downarrow0$ in
\eqref{eq:Omegaplus-bounds}, we obtain
\[
\lim_{\varepsilon\downarrow0}\frac{|\Omega_+(\varepsilon)|}{\varepsilon}
=\frac{1}{\|a\|}\int_{H(a)\cap \Omega,\ x_1\ge0} x_1\,d\sigma(x).
\]
A completely similar argument for $\Omega_-(\varepsilon)$ (considering $x_1\le0$)
gives
\[
\lim_{\varepsilon\downarrow0}\frac{|\Omega_-(\varepsilon)|}{\varepsilon}
=\frac{1}{\|a\|}\int_{H(a)\cap \Omega,\ x_1\le0} x_1\,d\sigma(x).
\]
Subtracting the two limits, we obtain
\[
\frac{\partial V}{\partial a_1}(a)
=\lim_{\varepsilon\to0}\frac{D_1(\varepsilon)}{\varepsilon}
=\frac{1}{\|a\|}\int_{H(a)\cap \Omega} x_1\,d\sigma(x).
\]

Repeating the same computation for all coordinate directions $e_j$ shows that
\begin{equation}\label{eq:gradV-moment}
    \nabla V(a)
=\frac{1}{\|a\|}\int_{H(a)\cap \Omega} x\,d\sigma(x).
\end{equation}
Thus, $V$ is $C^1$ in a neighborhood of $a$.

\medskip

\emph{Step 2: Centroid formula.}

Taking the inner product of \eqref{eq:gradV-moment} with $a$ and using
$\ip{a}{x}=1$ on $H(a)$, we obtain
\[
\ip{a}{\nabla V(a)}
=\frac{1}{\|a\|}\int_{H(a)\cap \Omega} \ip{a}{x}\,d\sigma(x)
=\frac{1}{\|a\|}\int_{H(a)\cap \Omega} 1\,d\sigma(x).
\]
Hence
\[
\int_{H(a)\cap \Omega} x\,d\sigma(x)
=\|a\|\,\nabla V(a),\qquad
\int_{H(a)\cap \Omega} 1\,d\sigma(x)
=\|a\|\,\ip{a}{\nabla V(a)}.
\]
Dividing these two identities gives
\[
x(a)
=\frac{\nabla V(a)}{\ip{a}{\nabla V(a)}}.
\]
\end{proof}

\subsection{Floating body characterization}

We next adapt Lemma~7 of \cite{Meyer}.

\begin{lemma}\label{lem:7-adapted}
Let $\Omega\subset\R^{n+1}$ and $B\subset \Omega$ both be unbounded strictly convex sets with $0\notin\operatorname{int}B$. Assume that for every support hyperplane $H$ of $B$, the centroid of $H\cap \Omega$ exists and belongs to $B$. Suppose further that for every $a\in\R^{n+1}\setminus\{0\}$ such that $H(a)$ is a support hyperplane of $B$ and $B\subset C(a)^+$, we have
\[
0<V(a)<\infty,\qquad V(a):=\mathcal{H}^{n+1}(\Omega\cap C(a)^+).
\]
Then there exists a constant $c>0$ such that
\[
V(a)=c\qquad\text{for all such }a.
\]
\end{lemma}

\noindent\textbf{Geometric preliminaries.}
Let $a$ be as in the statement, so that $H(a)$ supports $B$ and $B\subset C(a)^+$. We record two consequences of strict convexity of $B$.

\begin{itemize}
  \item[(i)] \emph{Unique contact point.} Strict convexity of $B$ implies that the support hyperplane $H(a)$ meets $B$ at a unique point, denoted $b(a)\in\partial B$. Thus $H(a)\cap B=\{b(a)\}$ and $\ip{a}{b(a)}=1$.
  \item[(ii)] \emph{Tangent space to the support parameter set.} Let
        \[
        \mathcal{S}:=\{a\in\R^{n+1}\setminus\{0\}:\ B\subset C(a)^+,\ H(a)\text{ supports }B\}.
        \]
        Since $B=\bigcap_{a\in\mathcal{S}}C(a)^+$ and the active inequality at $a$ is $\ip{a}{b(a)}\ge 1$ with equality at the unique contact point, the tangent space of $\mathcal{S}$ at $a$ satisfies
        \begin{equation}\label{eq:tangent-orthogonality}
        T_a\mathcal{S}\subset\{v\in\R^{n+1}:\ \ip{v}{b(a)}=0\}.
        \end{equation}
        In other words, $b(a)$ is normal to the hypersurface of parameters realizing the same support point.
\end{itemize}

\begin{proof}[Proof of Lemma~\ref{lem:7-adapted}]
Fix $a\in\mathcal{S}$ with $0<V(a)<\infty$. By hypothesis, the centroid of $H(a)\cap \Omega$ lies in $B$, and since it also lies on $H(a)$ we must have
\begin{equation}\label{eq:centroid-equals-contact}
x(a)=b(a).
\end{equation}
Indeed, by strict convexity $H(a)\cap B=\{b(a)\}$, so the only point of $B$ lying on $H(a)$ is $b(a)$.

By Lemma~\ref{gradV} and \eqref{eq:centroid-equals-contact}, the gradient of $V$ at $a$ is parallel to $b(a)$:
\begin{equation}\label{eq:grad-parallel}
\nabla V(a)=\lambda(a)\,b(a)\qquad\text{with}\quad \lambda(a):=\ip{a}{\nabla V(a)}\neq 0.
\end{equation}
Let $v\in T_a\mathcal{S}$. Using \eqref{eq:tangent-orthogonality} and \eqref{eq:grad-parallel},
\[
D_v V(a)=\ip{\nabla V(a)}{v}=\lambda(a)\,\ip{b(a)}{v}=0.
\]
Thus all directional derivatives of $V$ along $T_a\mathcal{S}$ vanish at $a$. Since $a$ was arbitrary, $V$ is locally constant on $\mathcal{S}$.

The set $\mathcal{S}$ of support parameters for a strictly convex body is connected (because the Gauss map of $\partial B$ is continuous and surjective onto the set of outer unit normals). Hence local constancy implies global constancy: there exists $\alpha>0$ such that
\[
V(a)=\alpha\qquad\text{for all }a\in\mathcal{S}\text{ with }0<V(a)<\infty.
\]
This is exactly the desired conclusion.
\end{proof}

\begin{proof}[Proof of Theorem~\ref{thm:main}]
An affine hypersphere has all affine normal lines either
\begin{enumerate}[label=(\roman*),ref=(\roman*)]  \item\label{case:parallel} parallel to some vector $e_{n+1}$; or
  \item\label{case:concurrent} concurrent through a single point.
\end{enumerate}

In both cases \ref{case:parallel} and \ref{case:concurrent}, we may construct $B$ in the following way:

\[
B =
\begin{cases}
\Omega + \lambda e_{n+1} & \text{in case \ref{case:parallel}},\\
\lambda\,\Omega       & \text{in case \ref{case:concurrent}},
\end{cases}
\]

For each support hyperplane $P_u$ of $B$ with normal $u\in U$, the contact point $x_0$ between $P_u$ and $B$ can be written as
\[
x_0=
\begin{cases}
N^{-1}(u)+\lambda e_{n+1} & \text{in case \ref{case:parallel}},\\[0.3em]
\lambda\, N^{-1}(u)   & \text{in case \ref{case:concurrent}},
\end{cases}
\]

for some $\lambda>0$ in case \ref{case:parallel} or $\lambda>1$ in case \ref{case:concurrent}, where $N^{-1}(u)$ denotes the point of $\partial\Omega$ with outer normal $u$. In particular, the centroid of the section $P_u\cap\Omega$ lies on the centroid line $\ell_u^\Omega$, hence lies in $B$ by construction of $B$ as a suitable translate or scaling of $\Omega$.

Thus the hypotheses of Lemma~\ref{lem:7-adapted} are satisfied with $\Omega$ and $B$ as above. We conclude that there exists a constant $\alpha>0$ such that $V(a)=\alpha$ for all support parameters $a$ corresponding to $u\in N(\partial \Omega)$.

We recall a theorem which summarizes results from Theorem~2 of~\cite{kimHyp}, and Theorem~5 of~\cite{kimPara}.

\begin{theorem}
    Let $M$ be a smooth convex hypersurface in $\R^{n+1}$ defined by the graph of the function $f:\R^n\mapsto\R$ and let $\Phi(x)$ be the plane tangent to $f(x)$. For $k\ge 0$, denote by $V_x(k)$ the volume of the region bounded between $\Phi(x)+k$ and $M$. For $k\ge 1$, denote by $V^*_x(k)$ the volume of the region bounded between $k\ \Phi(x)$ and $M$. If $V_x(k)$ is constant for every $x\in \R^n$, then $M$ is an elliptic paraboloid. If $V^*_x(k)$ is constant for every $x\in \R^n$, then $M$ is one sheet of a two-sheeted hyperboloid.
\end{theorem}

Therefore, in case~\ref{case:parallel} $\partial\Omega$ must be an elliptic paraboloid, and in case~\ref{case:concurrent} $\partial\Omega$ is one sheet of a two-sheeted hyperboloid.

Finally, if $\partial\Omega$ is closed and bounded, the conclusion that $\partial\Omega$ is an ellipsoid is exactly the Meyer--Reisner theorem \cite{Meyer}.
\end{proof}

Notice that the proof relies on the centroid lines being either parallel or concurrent, and hence on
$\partial\Omega$ being an affine hypersphere. In fact, since a compact $\partial\Omega$ has all centroid lines intersecting at the centroid of $\Omega$, the
classification of compact affine hyperspheres implies that $\partial\Omega$ must be an ellipsoid.
By the same classification, when all centroid lines are parallel, $\partial\Omega$ must be a paraboloid. Therefore, the above result is most interesting in the hyperbolic case.\\

A resolution to question~\ref{ques:question} would follow from whether a convex hypersurface with SCCP is necessarily an affine hypersphere; this reduction is the motivation for the next sections.

\section{Asymptotic behavior of centroid lines and recession cones}\label{sec:extension}

In this section we analyze the asymptotic behavior of centroid lines. We find that for strictly convex sets with SCCP, centroid lines of $\Omega$ and of its recession cone $\operatorname{rec}(\Omega)$ have the same direction. If in addition $\partial\Omega$ is asymptotic to $\operatorname{rec}(\Omega)$, then the centroid lines of $\Omega$ and of its recession cone $\operatorname{rec}(\Omega)$ are the same. $\partial\Omega$ must then be an affine hypersphere and therefore a quadric. 

\begin{lemma}\label{lem:homothety}
Let $\Omega\subset\R^{n+1}$ be a strictly convex, unbounded set, and assume that $\Omega$ has SCCP. Let $\mathcal{C}:=\operatorname{rec}(\Omega)$ be its recession cone. Fix $u\in U:=N(\partial\Omega)$, and denote by $\ell_u^\Omega$ the centroid line associated to $\Omega$ and by $\ell_u^{\mathcal{C}}$ the centroid line associated to $\mathcal{C}$. Then
\[
\ell_u^\Omega\parallel \ell_u^{\mathcal{C}}.
\]
\end{lemma}

\begin{proof}
By SCCP, $\ell_u$ depends affinely on $t$:
\[
\ell_u(t)=q_u+t\,w_u,\qquad t\in I(u),
\]
for some $q_u,w_u\in\R^{n+1}$. 

For the cone $\mathcal{C}=\operatorname{rec}(\Omega)$, whenever the sections $\Sigma(u,t;\mathcal{C})$ have finite measure, their centroids lie on a line $\ell_u^{\mathcal{C}}$ passing through the apex of $\mathcal{C}$. Since $\mathcal{C}$ is a cone, the sections $\Sigma(u,t;\mathcal{C})$ are homothetic as $t$ varies, and their centroids scale linearly:
\[
\ell_u^{\mathcal{C}}(t)=t\,w_u^{\mathcal{C}},
\]
for some $w_u^{\mathcal{C}}\in\R^{n+1}$.

Now, let $R>0$ and consider the scaled sets
\[
\Omega_R:=\frac{1}{R}\Omega.
\]
Since $\Omega$ is closed and convex, the asymptotic cone
\[
\Omega_\infty:=\lim_{R\to\infty}\frac{1}{R}\Omega
\]
exists and coincides with the recession cone:
\[
\Omega_\infty = \operatorname{rec}(\Omega) = \mathcal{C}.
\]
In particular, for any fixed $t>0$ and $u\in N(\partial\Omega)$,
\[
\Omega_R\cap\Pi(u,t) \;\xrightarrow[R\to\infty]{d_H}\; \mathcal{C}\cap\Pi(u,t)
\]
in Hausdorff distance inside the hyperplane $\Pi(u,t)$. Since each slice is a bounded convex subset of $\Pi(u,t)$ with positive finite $n$--measure, convergence in Hausdorff distance implies convergence of centroids:
\[
\ell_u^{\Omega_R}(t)\xrightarrow[R\to\infty]{}\ell_u^\mathcal{C}(t)
=t\,w_u^{\mathcal{C}}.
\]

On the other hand, the section of $\Omega_R$ at level $t$ is just a rescaled section of $\Omega$ at level $Rt$:
\[
\Omega_R\cap\Pi(u,t)
=\frac{1}{R}\bigl(\Omega\cap\Pi(u,Rt)\bigr),
\]
so its centroid is
\[
\ell_u^{\Omega_R}
=\frac{1}{R}\,\ell_u(Rt)
=\frac{1}{R}\bigl(q_u+Rt\,w_u\bigr)
=\frac{1}{R}q_u + t\,w_u.
\]
Letting $R\to\infty$ we obtain
\[
\ell_u^{\Omega_R}\xrightarrow[R\to\infty]{} t\,w_u.
\]

Combining the two limits, we get
\[
t\,w_u = t\,w_u^{\mathcal{C}}\qquad\text{for all }t>0,
\]
hence $w_u=w_u^{\mathcal{C}}$. Therefore the direction vector of the centroid line $\ell_u^\Omega$ coincides with that of $\ell_u^{\mathcal{C}}$, and
\[
\ell_u^\Omega\parallel \ell_u^{\mathcal{C}}. \qedhere
\]
\end{proof}

\begin{proof}[Proof of Theorem~\ref{thm:thm2}]
If $\operatorname{rec}(\Omega)=\{0\}$, then $\Omega$ is bounded, and the Meyer--Reisner theorem \cite{Meyer} implies that $\partial\Omega$ is an ellipsoid. 
Hence assume $\operatorname{rec}(\Omega)\neq\{0\}$, and write $\mathcal C:=\operatorname{rec}(\Omega)$.

By Lemma~\ref{lem:homothety}, for each admissible normal direction $u\in N(\partial\Omega)$, the centroid line $\ell_u^\Omega$ is parallel to the corresponding centroid line $\ell_u^{\mathcal C}$ of the cone $\mathcal C$.

\smallskip

\noindent\emph{Case 1: $\dim\mathcal C=1$.}
Then $\mathcal C$ is a ray, so every line $\ell_u^{\mathcal C}$ has the same direction. 
Consequently all centroid lines $\ell_u^\Omega$ are parallel. This means the affine normal lines of $\partial\Omega$ are parallel as well; thus $\partial\Omega$ is an improper affine hypersphere. 
Applying Theorem~\ref{thm:main} yields that $\partial\Omega$ is an elliptic paraboloid.

\smallskip

\noindent\emph{Case 2: $\dim\mathcal C=n+1$ and $\partial\Omega$ is asymptotic to $\mathcal C$.}
Fix $u\in N(\partial\Omega)$. 
As $t\to\infty$, asymptotic convergence of $\partial\Omega$ to $\partial\mathcal C$ implies that the bounded sections $\Sigma(u,t;\Omega)$ converge (in Hausdorff distance inside $\Pi(u,t)$) to $\Sigma(u,t;\mathcal C)$. 
In particular, the centroids satisfy
\[
\ell_u^\Omega(t)-\ell_u^\mathcal{C}(t)\to 0
\qquad (t\to\infty).
\]
This forces $\ell_u^\Omega=\ell_u^{\mathcal C}$. 
Therefore all centroid lines $\ell_u^\Omega$ are concurrent at the vertex of $\mathcal C$, meaning the affine normals of $\partial\Omega$ are concurrent, so $\partial\Omega$ is a proper affine hypersphere. 
Applying Theorem~\ref{thm:main} yields that $\partial\Omega$ is a sheet of a two-sheeted hyperboloid.
\end{proof}

\section{Conclusion and open problems}

Settling Question~\ref{ques:question} amounts to answering whether the following dichotomy is true:

\medskip

\noindent
Let $\Omega\subset\R^{n+1}$ be connected with smooth strictly convex boundary $\partial\Omega$ and satisfying SCCP. Does necessarily
\[
\bigcap_{u\in N(\partial\Omega)}\ell_u^\Omega = \{p\}
\]
for some $p\in\R^{n+1}$, or else
\[
\ell_{u_1}^\Omega\parallel \ell_{u_2}^\Omega\qquad\forall u_1,u_2\in U \text{ ?}
\]
There are four subcases, two of which are settled, while two remain open.

\subsection*{Known cases}

\paragraph{Case 1: full-dimensional recession cone, $\Omega$ asymptotic to its cone.}
Assume that $\Omega$ has an $(n+1)$--dimensional recession cone $\mathcal{C}=\operatorname{rec}(\Omega)$ and that $\partial\Omega$ is asymptotic to $\mathcal{C}$. By Theorem~\ref{thm:thm2}, for each $u\in U$ the centroid line $\ell_u^\Omega$ coincides with the centroid line of $\mathcal{C}$, so
\[
\bigcap_{u\in N(\partial\Omega)}\ell_u^\Omega = \{p\},
\]
where $p$ is the vertex of $\operatorname{rec}(\Omega)$. In this case $\partial\Omega$ is asymptotic to a cone and shares its centroid lines.

\paragraph{Case 2: one-dimensional recession cone.}
If $\dim\operatorname{rec}(\Omega)=1$, then $\operatorname{rec}(\Omega)$ is a ray, 
\[
\ell_{u_1}^\Omega\parallel \ell_{u_2}^\Omega\qquad\forall u_1,u_2\in U,
\]
and by Theorem~\ref{thm:thm2} $\partial\Omega$ is a paraboloid.

\subsection*{Open cases}

The remaining two cases are, to our knowledge, open and contain the core difficulty of classifying convex hypersurfaces with SCCP.

\paragraph{Case 3: full-dimensional recession cone, not asymptotic to the cone.}
Assume that $\operatorname{rec}(\Omega)$ has full dimension $(n+1)$ but that $\partial\Omega$ is \emph{not} asymptotic to $\partial\operatorname{rec}(\Omega)$. By Cheng-Yau's solution of the Calabi conjecture, there exists a hyperbolic affine hypersphere $Y$ asymptotic to $\operatorname{rec}(\Omega)$, whose affine normal at the point with Euclidean normal $u$ is parallel to $\ell_u^\Omega$. The crucial question is whether this forces $\partial\Omega$ itself to be an affine hypersphere. If so, Theorem~\ref{thm:thm2} would apply and the classification would be complete in this case.

\paragraph{Case 4: intermediate-dimensional recession cone.}
Finally, suppose that $\operatorname{rec}(\Omega)$ has dimension $m$ with $1<m<n+1$. In this situation the SCCP implies that all centroid lines $\ell_u^\Omega$ are parallel to a fixed $m$--dimensional subspace $V$, but neither concurrency at a single point nor mutual parallelism of all centroid lines is automatic.

\section*{Acknowledgments}
I thank Jacob Ogden for his generous mentorship throughout this project. Our long and insightful meetings, his continued support, and his contributions were invaluable.

\bibliographystyle{plain}
\bibliography{mybib}

@article {Meyer,
    AUTHOR = {Meyer, M. and Reisner, S.},
     TITLE = {Characterizations of ellipsoids by section-centroid location},
   JOURNAL = {Geom. Dedicata},
  FJOURNAL = {Geometriae Dedicata},
    VOLUME = {31},
      YEAR = {1989},
    NUMBER = {3},
     PAGES = {345--355},
      ISSN = {0046-5755,1572-9168},
   MRCLASS = {52A20},
  MRNUMBER = {1025195},
MRREVIEWER = {Peter\ M.\ Gruber},
       DOI = {10.1007/BF00147465},
       URL = {https://doi.org/10.1007/BF00147465},
}

@article {kimPara,
    AUTHOR = {Kim, Dong-Soo and Kim, Young Ho},
     TITLE = {Some characterizations of spheres and elliptic paraboloids
              {II}},
   JOURNAL = {Linear Algebra Appl.},
  FJOURNAL = {Linear Algebra and its Applications},
    VOLUME = {438},
      YEAR = {2013},
    NUMBER = {3},
     PAGES = {1356--1364},
      ISSN = {0024-3795,1873-1856},
   MRCLASS = {53A05 (52A22 53A07)},
  MRNUMBER = {2997816},
MRREVIEWER = {Graciela\ Mar\'ia\ Desideri},
       DOI = {10.1016/j.laa.2012.08.024},
       URL = {https://doi.org/10.1016/j.laa.2012.08.024},
}

@article {kimHyp,
    AUTHOR = {Kim, Dong-Soo},
     TITLE = {Ellipsoids and elliptic hyperboloids in the {E}uclidean space
              {${\Bbb E}^{n+1}$}},
   JOURNAL = {Linear Algebra Appl.},
  FJOURNAL = {Linear Algebra and its Applications},
    VOLUME = {471},
      YEAR = {2015},
     PAGES = {28--45},
      ISSN = {0024-3795,1873-1856},
   MRCLASS = {53A07 (52A20 52A38)},
  MRNUMBER = {3314322},
MRREVIEWER = {Alina\ Stancu},
       DOI = {10.1016/j.laa.2014.12.014},
       URL = {https://doi.org/10.1016/j.laa.2014.12.014},
}

@incollection{Calabi1972,
  author    = {Calabi, Eugenio},
  title     = {Complete affine hyperspheres {I}},
  booktitle = {Symposia Mathematica},
  volume    = {10},
  pages     = {19--38},
  publisher = {Istituto Nazionale di Alta Matematica, Academic Press},
  year      = {1972},
}

@book{LiSimonZhao1993,
  author    = {Li, An-Min and Simon, Udo and Zhao, Guosong},
  title     = {Global affine differential geometry of hypersurfaces},
  series    = {De Gruyter Expositions in Mathematics},
  volume    = {11},
  publisher = {Walter de Gruyter},
  year      = {1993},
}

@book{AuslenderTeboulle2003,
  author    = {Alfred Auslender and Marc Teboulle},
  title     = {Asymptotic Cones and Functions in Optimization and Variational Inequalities},
  series    = {Springer Monographs in Mathematics},
  publisher = {Springer},
  address   = {New York},
  year      = {2003},
  isbn      = {0-387-95520-8},
  doi       = {10.1007/b97594}
}

@article{GarciaGoicocheaLopezMartinez2025,
  author  = {Yboon Garc\'ia and Bruno Goicochea and Rub\'en L{\'o}pez and Javier Mart{\'\i}nez},
  title   = {Horizon and Recession Asymptotic Notions for Sets and Mappings: {A} Unified Approach},
  journal = {Journal of Optimization Theory and Applications},
  year    = {2025},
  volume  = {205},
  number  = {2},
  pages   = {33},
  doi     = {10.1007/s10957-025-02655-y}
}

@book{NomizuSasaki1994,
  author    = {Nomizu, Katsumi and Sasaki, Takashi},
  title     = {Affine Differential Geometry: Geometry of Affine Immersions},
  publisher = {Cambridge University Press},
  year      = {1994}
}

@article{Tzitzeica1907,
  author  = {Tzitzeica, Gheorghe},
  title   = {Sur une nouvelle classe de surfaces},
  journal = {Comptes Rendus Hebdomadaires des S\'eances de l'Acad\'emie des Sciences},
  volume  = {144},
  pages   = {1257--1259},
  year    = {1907}
}

@article{Tzitzeica1908,
  author  = {Tzitzeica, Gheorghe},
  title   = {Sur une nouvelle classe de surfaces},
  journal = {Rend. Circ. Mat. Palermo},
  volume  = {25},
  pages   = {180--187},
  year    = {1908}
}

@article{Tzitzeica1909,
  author  = {Tzitzeica, Gheorghe},
  title   = {Sur une nouvelle classe de surfaces, 2\`eme partie},
  journal = {Rend. Circ. Mat. Palermo},
  volume  = {25},
  pages   = {210--216},
  year    = {1909}
}

@book{Blaschke1923,
  author    = {Blaschke, Wilhelm},
  title     = {Vorlesungen {\"u}ber Differentialgeometrie II: Affine Differentialgeometrie},
  publisher = {Springer-Verlag},
  address   = {Berlin},
  year      = {1923}
}

@article{Calabi1972CompleteAffineHyperspheresI,
  author  = {Calabi, Eugenio},
  title   = {Complete affine hyperspheres. {I}},
  journal = {Symposia Mathematica},
  volume  = {10},
  pages   = {19--38},
  year    = {1972}
}

@article{ChengYau1986CompleteAffineHyperspheresI,
  author  = {Cheng, Shiu-Yuen and Yau, Shing-Tung},
  title   = {Complete affine hyperspheres. Part {I}. The completeness of affine metrics},
  journal = {Communications on Pure and Applied Mathematics},
  volume  = {39},
  number  = {6},
  pages   = {839--866},
  year    = {1986}
}

@article{ChengYau1977MongeAmpereRegularity,
  author  = {Cheng, Shiu-Yuen and Yau, Shing-Tung},
  title   = {On the regularity of the {M}onge-{A}mp{\`e}re equation $\det((\partial^2 u/\partial x_i \partial x_j))=F(x,u)$},
  journal = {Communications on Pure and Applied Mathematics},
  volume  = {30},
  pages   = {41--68},
  year    = {1977}
}

\noindent {Department of Mathematics, Padelford Hall\\
University of Washington, Seattle, WA 98195
}\\
Email: \texttt{aboren@uw.edu}

\end{document}